\documentclass{amsart}
\font\fiverm=cmr5
\font\ehsc=cmcsc10 scaled 850

\let\\=\backslash
\let\sse=\subseteq
\let\noi=\noindent

\let\limply=\Longrightarrow

\def\0{\{0\}}
\def\span{{\kern.5pt{\rm span}\kern1pt}}
\def\conv{{\;\longrightarrow\;}}
\def\wconv{{{\buildrel_{\scriptstyle w}\over\conv}}}
\def\sconv{{{\buildrel_{\scriptstyle s}\over\conv}}}
\def\query{{{\buildrel_{^{\scriptstyle ?}}\over\limply}}}

\def\sslash{\hbox{{\fiverm/}}}
\def\notwconv{{{\wconv\kern-13pt\sslash}\kern9pt}}
\def\notsconv{{{\sconv\kern-13pt\sslash}\kern9pt}}

\def\B{{\mathcal B}}
\def\C{{\kern.5pt\mathcal C}}
\def\Oe{{\mathcal O}}
\def\X{{\mathcal X}}

\def\BX{{\B[\X]}}

\def\CC{{\mathbb C\kern.5pt}}
\def\FF{{\mathbb F\kern.5pt}}
\def\NN{{\mathbb N\kern.5pt}}
\def\RR{{\mathbb R\kern.5pt}}

\newsymbol\varnothing 203F
\let\void=\varnothing

\def\newmatrix#1{\null\,\vcenter{
			\baselineskip=8pt\mathsurround=-0pt\ialign{
			\hfil ${##}$
			\hfil &&
			\hfil ${##}$
			\hfil \crcr
			\mathstrut \crcr
			\noalign{\kern-\baselineskip}#1 \crcr
			\mathstrut \crcr
			\noalign{\kern-\baselineskip} \crcr }}\!}

\newtheorem{theorem}{Theorem}
\newtheorem{lemma}{Lemma}
\newtheorem{corollary}{Corollary}
\newtheorem{proposition}{Proposition}

\theoremstyle{definition}
\newtheorem{definition}{Definition}
\newtheorem{remark}{Remark}

\numberwithin{theorem}{section}
\numberwithin{lemma}{section}
\numberwithin{corollary}{section}
\numberwithin{proposition}{section}
\numberwithin{conjecture}{section}
\numberwithin{definition}{section}
\numberwithin{remark}{section}
\numberwithin{question}{section}

\begin{document}

\vglue-75pt\noindent
\hfill{\it Mathematical Proceedings of the Royal Irish Academy}\/,
{\bf 120A} (2020) 7--18

\vglue15pt
\title[Denseness of supercyclic vectors]
      {Denseness of sets of supercyclic vectors}
\author[C.S. Kubrusly]{C.S. Kubrusly}
\address{Applied Mathematics Department, Federal University,
         Rio de Janeiro, Brazil}
\email{carloskubrusly@gmail.com}
\renewcommand{\keywordsname}{Keywords}
\date{November 24, 2019}

\begin{abstract}
The sets of strongly supercyclic, weakly l-sequentially supercyclic, weakly
sequentially supercyclic, and weakly supercyclic vectors for an arbitrary
normed-space operator are all dense in the normed space, regardless the
notion of denseness one is considering, provided they are nonempty.
\end{abstract}

\maketitle

%%%%%%%%%%%%%%%%%%%%%%%%%%%%%%%%%%%%%%%%%%%%%%%%%%%%%%%%%% SECTION 1
\vskip-15pt\noi
\section{Introduction}

Strong supercyclicity is an important topic in operator theory for some
decades already (see, e.g., \cite{AB, BM2, BBP, Her, Hez, Sal})$.$ Several
forms of weak supercyclicity have recently been investigated as well
(see, e.g., \cite{BM, BM1, Dug, Kub, KD1, KD2, MS, San1, San2, Shk}).

\vskip4pt
Supercyclicity for normed-space operators means denseness of a projective
orbit$.$ Associated with the notion of denseness one is considering there
corresponds the concepts of strongly supercyclic, weakly l-sequentially
supercyclic, weakly sequentially supercyclic, and weakly supercyclic vectors.

\vskip4pt
This paper focus on another question on denseness, viz., denseness of sets of
supercyclic vectors$.$ Besides common notions of denseness in the weak and
norm topologies, intermediate notions of weak sequential and weak
l-sequential denseness are considered$.$ The main result appears in
Theorem 5.1 which leads in Corollary 6.1 to the following consequence$:$ if
any of the above sets of supercyclic vectors is nonempty, then it is dense
with respect to any notion of denseness$.$ In particular, the set of weakly
l-sequentially supercyclic vectors is norm dense (i.e., strongly dense),
which is a useful improvement over previously known results along this line.

\vskip4pt
The paper gather results on the above mentioned four forms of denseness for
sets of {\it supercyclic}\/ vectors (regarding all four forms of
supercyclicity) into a concise statement, with emphasis on the set of weakly
l-sequentially supercyclic vectors.

%%%%%%%%%%%%%%%%%%%%%%%%%%%%%%%%%%%%%%%%%%%%%%%%%%%%%%%%%% SECTION 2
\section{Preliminary Notions}

Let $\FF$ stand either for the complex field $\CC$ or for the real field
$\RR$, and let $\X$ be an infinite-dimensional normed space over $\FF$$.$ An
$\X$-valued sequence $\{x_n\}$ is {\it strongly convergent}\/ if there is an
${x\in\X}$ such that ${\|x_n-x\|\to0}$ (notation$:$ ${x_n\!\sconv x}$ or
$x={s\hbox{\,-}\lim x_n}$), and it is {\it weakly convergent}\/ if there is
an ${x\in\X}$ such that ${|f(x_n-x)|\to0}$ for every $f$ in the dual space
$\X^*$ of $\X$ (notation$:$ ${x_n\!\wconv x}$ or $x={w\hbox{\,-}\lim x_n}).$
Strong convergence trivially implies weak convergence (to the same limit).

\vskip4pt
Subsets of $\X$ are {\it strongly closed}\/ or {\it weakly closed}\/ if
they are closed in the norm or weak topologies of $\X.$ {\it Strong
closure}\/ or {\it weak closure}\/ of a set $A$ is the smallest strongly
or weakly closed set that includes $A$ (i.e., the intersection of all
strongly or weakly closed sets including $A$ --- notation$:$ $A^-$ or
$A^{-w}).$ Thus $A$ is strongly closed or weakly closed if and only if
${A=A^-}$ or ${A=A^{-w}}.$ A set $A$ is {\it strongly dense}\/ or
{\it weakly dense}\/ in $\X$ if ${A^-\!=\X}$ or ${A^{-w}\!=\X}$.

%%%%%%%%%%%%%%%%%%%%%%%%%%% DEFINITION 2.1
\vskip4pt\noi
\begin{definition}
Let $A$ be a subset of $\X$.
\begin{description}
\vskip4pt
\item$\kern-5pt$(a)
The set $A$ is {\it weakly sequentially closed}\/ if every $A$-valued weakly
convergent sequence has its limit in $A$ (i.e., if
$\{{x=w\hbox{\,-}\lim x_n}$ with ${x_n\in A}$ $\limply$ ${x\in A}\}$).
\vskip4pt
\item$\kern-5pt$(b)
The {\it weak sequential closure}\/ of $A$ is the smallest weakly sequentially
closed set that incudes $A$ (i.e., is the intersection of all weakly
sequentially closed sets including $A$) --- notation: $A^{-ws}$.
\vskip4pt
\item$\kern-5pt$(c)
The set $A$ is {\it weakly sequentially dense}\/ in $\X$ if ${A^{-ws}\!=\X}$.
\vskip4pt
\item$\kern-5pt$(d)
The {\it weak limit set}\/ $A^{-wl}\!$ of $A$ is the set of all weak limits of
weakly convergent $A$-valued sequences (i.e.,
$A^{-wl}=\{{x\in\X}\!:$ ${x=w\hbox{\,-}\lim x_n}$ with ${x_n\in A}\}$.
\vskip4pt
\item$\kern-5pt$(e)
The set $A$ is {\it weakly l-sequentially dense}\/ in $\X$ if ${A^{-wl}\!=\X}$.
\end{description}
\end{definition}

\vskip6pt
A collection of basic results required in the sequel is given below$.$
Most are either straightforward or well-known and standard$.$
We prove item (e) only.

%%%%%%%%%%%%%%%%%%%%%%%%%%% PROPOSITION 2.1
\vskip4pt\noi
\begin{proposition}
Consider the setup of Definition\/ $2.1$.
\begin{description}
\vskip4pt
\item{$\kern-9pt$\rm(a)}
$\,A$ is weakly closed $\,\limply\!$ $A$ is weakly sequentially closed
$\,\limply\!$ $A$ is strongly closed.
\vskip4pt
\item{$\kern-9pt$\rm(b)}
$\,A\sse A^-\sse A^{-wl}\sse A^{-ws}\sse A^{-w}\sse\X$.
\vskip4pt
\item{$\kern-9pt$\rm(c)}
$\,$Now consider the following assertions\/$.$
\vskip4pt
\begin{description}
\item{$\kern-9pt$\rm(c$_0$)}
$\,A$ is strongly dense\/ $($equivalently, for every\/ ${x\in\X}$ there exists
an\/ $A$-valued sequence\/ $\{x_n\}$ such that\/ ${x_n\!\sconv x}\kern1pt)$,
\vskip4pt
\item{$\kern-9pt$\rm(c$_1$)}
$\,A$ is weakly l-sequentially dense\/ $($equivalently, for every\/
${x\in\X}$ there exists an\/ $A$-valued sequence\/ $\{x_n\}$ such that\/
${x_n\!\wconv x}\kern1pt)$,
\vskip4pt
\item{$\kern-9pt$\rm(c$_2$)}
$A$ is weakly sequentially dense,
\vskip4pt
\item{$\kern-9pt$\rm(c$_3$)}
$A$ is weakly dense.
\end{description}
\vskip4pt
$\kern-2pt$They are related by this chain of implications\/$.$
\vskip4pt\noi
$\kern14pt$
{\rm(c$_0$)}
$\limply\!$
{\rm(c$_1$)}
$\limply\!$
{\rm(c$_2$)}
$\limply\!$
{\rm(c$_3$)}.
\vskip4pt
\item{$\kern-6pt$\rm(d)}
$\,A$ is convex $\,\limply A^-=A^{-wl}=A^{-ws}=A^{-w}\!$.
\vskip4pt
\item{$\kern-6pt$\rm(e)}
The following assertions are pairwise equivalent.
\vskip4pt\noi
{$\!$\rm(e$_1$)}
$A$ is weakly sequentially closed.
\vskip4pt\noi
{$\!$\rm(e$_2$)}
$A=A^{-ws}\!$.
\vskip4pt\noi
{$\!$\rm(e$_3$)}
$A=A^{-wl}\!$.
\end{description}
\end{proposition}

\vskip0pt
\begin{proof}
(e) Assertions $(e_1)$ and $(e_2)$ are trivially equivalent
by Definition 2.1(b)$.$ If $A$ is weakly sequentially closed, then
${A^{-wl}\kern-1pt=A}$ by Definitions 2.1(a,d) and so $(e_1)$ implies
$(e_3).$ Conversely, if ${A^{-wl}\kern-1pt=A}$, equivalently, if
$A^{-wl}\kern-1pt=\{{x\in\X}\!:$ ${x=w\hbox{\,-}\lim x_n}$ with
${x_n\in A}\}\sse A$, then
$\{{x=w\hbox{\,-}\lim x_n}$ with ${x_n\in A}$ $\limply$ ${x\in A}\}$,
and hence $A$ is weakly sequentially closed by Definition 2.1(a)$.$ Thus
$(e_3)$ implies $(e_1)$.
\end{proof}

\vskip4pt
Although it may happen $A^{-wl}\subset A^{-ws}$ (proper inclusion) in
Proposition 2.1(b), this is not the case if $A$ is weakly sequentially
closed by Proposition 2.1(e).

%%%%%%%%%%%%%%%%%%%%%%%%%%%%%%%%%%%%%%%%%%%%%%%%%%%%%%%%%% SECTION 3
\section{Supercyclic Vectors}

Let $\BX$ be the normed algebra of all bounded linear operators of a normed
space $\X$ into itself$.$ Given an operator ${T\kern-1pt\in\kern-1pt\BX}$
consider its power sequence $\{T^n\}_{n\ge0}.$ The orbit $\Oe_T(y)$ or
${\rm Orb}\kern1pt(T\kern-1pt,y)$ of a vector ${y\in\X}$ under an operator
${T\kern-1pt\in\kern-1pt\BX}$ is the set
$$
\Oe_T(y)={\bigcup}_{n\ge0}T^ny=\big\{T^ny\in\X\!:\,n\in\NN_0\big\}
$$
where $\NN_0$ denotes the set of nonnegative integers, and we write
${\bigcup}_{n\ge0}T^ny$ for the set
${\bigcup}_{n\ge0}T^n(\{y\})={\bigcup}_{n\ge0}\{T^ny\}.$ The orbit $\Oe_T(A)$
of a set ${A\sse\X}$ under $T$ is likewise defined$:$
$\Oe_T(A)=\bigcup_{n\ge0}T^n(A)\!=\!\bigcup_{z\in A}\Oe_T(z).$ Let
$[x]=\span\{x\}$ stand for the sub\-space of $\X$ spanned by a singleton
$\{x\}$ at a vector ${x\in\X}$, which is a one-dimensional subspace of $\X$
whenever $x$ is nonzero$.$ The projective orbit of a vector ${y\in\X}$ under
an operator ${T\in\BX}$ is the orbit of span of $\{y\}$; that is, the orbit
$\Oe_T([y])$ of $[y]$:
$$
\Oe_T([y])
=\Oe_T(\span\{y\})
={\bigcup}_{z\in[y]}\Oe_T(z)
=\big\{\alpha T^ny\in\X\!:\;\alpha\in\FF,\;n\in\NN_0\big\}.
$$

\vskip2pt
{\it Supercyclicity}\/ means denseness of projective orbits (if denseness
holds)$.$ 

\vskip6pt\noi
{\it Note}\/$.$
Clearly, $\Oe_T(\span\{y\})$ and $\span\Oe_T(y)$ are different sets$.$
Denseness of the latter is referred to as cyclicity, and denseness of the
orbit itself is referred to as hypercyclicity$.$ These will not be addressed
in this paper$.$ (For a brief discussion on them see, e.g.,
\cite[Sections 2 and 3]{KD1}; for a thorough view see, e.g., \cite{BM2} and
\cite{GP}$.$)

%%%%%%%%%%%%%%%%%%%%%%%%%%% DEFINITION 3.1
\vskip4pt\noi
\begin{definition}
$\!$Let ${T\kern-1pt\in\BX}$ be an operator on a normed space $\X$.
\begin{description}
\vskip4pt
\item$\kern-5pt$(a)
A vector ${y\kern-1pt\in\kern-1pt\X}$ is {\it strongly supercyclic}\/
(or {\it supercyclic}\/) for
$T$ if $\Oe_T([y])^-\!\!=\X$.
\vskip4pt
\item$\kern-5pt$(b)
A vector ${y\kern-1pt\in\kern-1pt\X}$ is {\it weakly l-sequentially
supercyclic}\/ for $T$ if $\Oe_T([y])^{-wl}\!=\X$.
\vskip4pt
\item$\kern-5pt$(c)
A vector ${y\kern-1pt\in\kern-1pt\X}$ is {\it weakly sequentially
supercyclic}\/ for $T$ if $\Oe_T([y])^{-ws}\!=\X$.
\vskip4pt
\item$\kern-5pt$(d)
A vector ${y\kern-1pt\in\kern-1pt\X}$ is {\it weakly supercyclic}\/
for $T$ if $\Oe_T([y])^{-w}\!=\X$.
\end{description}
\end{definition}

\vskip2pt
An operator $T$ is {\it strongly supercyclic}\/ (or simply
{\it supercyclic}\/), {\it weakly l-sequentially super\-cyclic}\/,
{\it weakly sequentially supercyclic}\/, or {\it weakly supercyclic}\/ if
there exists a strongly, weakly l-sequentially, weakly sequentially, or weakly
supercyclic vector $y$ for it$.$ Any form of cyclicity for an operator $T$
implies it acts on a separable space $\X.$ According to Proposition 2.1(c)
and Definition 3.1,
$$
\newmatrix{\!
_{_{\hbox{\ehsc strong}}}         & _{_{_{\textstyle\limply}}}\!\! &
_{\hbox{\ehsc weak l-sequential}} & _{_{_{\textstyle\limply}}}\!\! &
_{\hbox{\ehsc weak sequential}}   & _{_{_{\textstyle\limply}}}\!\! &
_{\hbox{\ehsc weak}}              & \!\!.                            \cr
\hbox{\ehsc supercyclicity}       &                                &
\hbox{\ehsc supercyclicity}       &                                &
\hbox{\ehsc supercyclicity}       &                                &
\hbox{\ehsc supercyclicity}       &                                  \cr}
$$
\vskip4pt\noi
The above implications are nonreversible (see, e.g.,
\cite[pp.38,39]{Shk}, \cite[pp.259,260]{BM2})$.$

\vskip4pt
A word on terminology$.$ Weak l-sequential supercyclicity was considered in
\cite{BCS} (and implicitly in \cite{BM1}), and it was referred to as weak
1-sequential supercyclicity in \cite{Shk}$.$ Although there are reasons for
such a terminology we have changed it here to weak l-sequential supercyclicity,
replacing the numeral ``1'' with the letter ``l'' for ``limit'' which better
describes the way this notion has been introduced here so far.

%%%%%%%%%%%%%%%%%%%%%%%%%%%%%%%%%%%%%%%%%%%%%%%%%%%%%%%%%% SECTION 4
\section{Auxiliary Notation, Terminology, and Auxiliary Results}

Take an arbitrary subset $A$ of the normed space $\X$ and set
\vskip4pt\noi
\begin{description}
\item{}
$A^{-0}=A^-$, the strong closure of $A$ (or the {\it 0-closure}\/ of $A$),
\vskip2pt
\item{}
$A^{-1}=A^{-wl}$, the weak limit set of $A$ (or the {\it 1-closure}\/ of $A$),
\vskip2pt
\item{}
$A^{-2}=A^{-ws}$, the weak sequential closure of $A$ (or the {\it 2-closure}\/
of $A$),
\vskip2pt
\item{}
$A^{-3}=A^{-w}$, the weak closure of $A$ (or the {\it 3-closure}\/ of $A$).
\end{description}
\vskip4pt\noi
Thus, according to Proposition 2.1(b),
$$
A\sse A^{-0}\sse A^{-1}\sse A^{-2}\sse A^{-3}\sse\X,
$$
and hence, following the denseness chain of Proposition 2.1(c),
$$
A^{-0}\!=\X\;\;\limply\;\;A^{-1}\!=\X\;\;\limply\;\;A^{-2}\!=\X
\;\;\limply\;\;A^{-3}\!=\X,
$$
where the notions of $k$-{\it denseness}\/ ($A^{-k}\kern-1pt=\X$) are in
general distinct for each $k={0, 1,2,3}.$ Accordingly, a set $A$ is
{\it 0-closed}\/, {\it 2-closed}\/ or {\it 3-closed}\/ if it is strongly
closed, weakly sequentially closed or weakly closed, respectively$.$ From
Proposition 2.1(a)
$$
\hbox{$A$ is $3$-closed $\;\limply$ $A$ is $2$-closed $\:\limply$
$A$ is $0$-closed},
$$
and from Proposition 2.1(e)
$$
\hbox{$A$ is $2$-closed}\; \iff A=A^{-1} \iff A=A^{-2}.
$$
{\it Strongly open}\/ (or {\it 0-open}\/) and {\it weakly open}\/ (or
{\it 3-open}\/) are sets which are open in the norm or in the weak topologies
of $\X$ (complements of strongly and weakly closed sets)$.$ A set $A$ is
{\it weakly sequentially open}\/ (or {\it 2-open}\/) if its complement is
weakly sequentially closed (i.e., ${A\sse\X}$ is $2$-open if ${\X\\A}$ is
$2$-closed)$.$ By Proposition 2.1(a)
$$
\hbox
{$A$ is $3$-open $\;\limply$ $A$ is $2$-open $\:\limply$ $A$ is $0$-open}.
$$
The norm topology and the weak topology are precisely the collections of all
\hbox{0-open} and 3-open sets, respectively$.$ Consider the collection of all
$2$-open (i.e., of all weakly sequentially open) subsets of $\X.$ This is
a topology on $\X$ as well.

%%%%%%%%%%%%%%%%%%%%%%%%%%% PROPOSITION 4.1
\vskip4pt\noi
\begin{proposition}
The collection of all\/ $2$-open sets is a topology on\/ $\X$.
\end{proposition}

\begin{proof}
{From} Definition 2.1(a)
$$
\hbox{A is $2$-closed} \iff
\big\{x=w\hbox{\;-}\lim x_n\;\,\hbox{with}\;\,x_n\in A \limply x\in A\big\}.
$$

\vskip2pt\noi
(a) {\it The empty set and the whole set are $2$-open}\/ (since they are
trivially $2$-closed).

\vskip6pt\noi
(b) Take a nonempty intersection ${\bigcap_\gamma A_\gamma}$ of $2$-closed
subsets $A_\gamma$ of $\X.$ Take any ${\bigcap_\gamma A_\gamma}$-valued weakly
convergent sequence $\{c_n\}.$ Since $A_\gamma$ are all $2$-closed, the weak
limit of $\{c_n\}$ lies in each $A_\gamma$, and so in
${\bigcap_\gamma A_\gamma}.$ So ${\bigcap_\gamma A_\gamma}$ is $2$-closed$.$
Thus ${\bigcup_\gamma A_\gamma}=$ ${\X\\\bigcap_\gamma A_\gamma}$ is
$2$-open$.$ Outcome$:$ {\it an arbitrary union of\/ $2$-open sets is\/
$2$-open}\/.

\vskip6pt\noi
(c) Consider the union ${A\cup B}$ of two nonempty $2$-closed subsets $A$
and $B$ of $\X.$ Take any ${A\cup B}$-valued weakly convergent (infinite)
sequence $\{c_n\}$, say
$$
c_n\wconv c\in\X
$$
(i.e., the $\FF$-valued sequence $\{f(c_n)\}$ converges in the metric space
${(\FF,|\cdot|)}$ to $f(c)$ in $\FF$ for every ${f\in\X^*}).$ If $\{c_n\}$
is eventually in one of the sets $A$ or $B$, then $c$ lies in such a set
(because both sets are $2$-closed) and so ${c\in A\cup B}.$ If $\{c_n\}$ is
not eventually in one of the sets, then it has infinitely many entries in $A$
and infinitely many entries in $B.$ Let $\{a_m\}$ and $\{b_m\}$ be
subsequences of $\{c_n\}$ whose entries are all in $A$ and all in $B$,
respectively$.$ Since the above displayed convergence takes place in the
metric space ${(\FF,|\cdot|)}$, every subsequence of $\{f(c_n)\}$ converges to
the same limit $f(c)$ for every ${f\in\X^*}\!.$ Then $\{a_m\}$ and $\{b_m\}$
converge weakly to $c$:
$$
a_m\wconv c\in\X
\quad\;\hbox{and}\;\quad
b_m\wconv c\in\X.
$$
Because $A$ and $B$ are $2$-closed, we get ${c\in A\cap B}.$ Hence
${c\in A\cup B}$ again$.$ Outcome$:$ the union of a pair of $2$-closed sets is
$2$-closed$.$ Thus (by induction) a finite union of $2$-closed sets is
$2$-closed, and so {\it a finite intersection of $2$-open sets is $2$-open}\/.
\end{proof}

\vskip4pt
Since a set $A$ is $0$-closed, $2$-closed, or $3$-closed if and only if
$A=A^{-0}\!$, $A=A^{-2}\!$, or $A=A^{-3}\!$, we say a set $A$ is
{\it 1-closed}\/ if ${A=A^{-1}}$ (and {\it 1-open}\/ if its complement is
$1$-closed)$.$ Such a definition of $1$-closedness collapses to the
definition of $2$-closedness since ${A\!=\!A^{-1}\!\!\iff\!\!A\!=\!A^{-2}}$
(Proposition 2.1(e))$.$ Even though ${A^{-1}\!\sse\!A^{-2}}\!$, if $A$ is not
$1$ or not $2$-closed, then $A^{-1}$ may be properly included in $A^{-2}\!$,
and the notion of $1$-denseness implies (but is not implied by) the notion
of $2$-denseness$.$ We refer to the weak limit set $A^{-1}$ of $A$ as the
$1$-closure of $A.$ This is an abuse of terminology since the map
${A\mapsto\!A^{-1}}$ is not a topological closure operation$.$ Strong, weak
sequential, and weak topologies are referred to as 0, 2, and 3-topologies,
respectively$.$ There is no 1-topology (and so $1$-closure and $1$-denseness
are not topological terminologies)$.$

\vskip4pt
The $k$-{\it interior}\/ of a set $A$, denoted $A^{\circ k}\!$, is the
interior of it regarding the respective notion of $k$-openness$:$ the largest
$k$-open set included in $A$ (i.e., the union of all $k$-open subsets of
$A).$ Since $1$-closedness coincides with $2$-closedness, the notions of
$1$-interior and $2$-interior coincide as well:
$A^{\circ 1}\!=A^{\circ 2}\!.$ Also, $({\X\\A})^{-k}\!={\X\\A}^{\circ k}$ and
$({\X\\A})^{\circ k}\!={\X\\A}^{-k}$ for $k={0,2,3}$ since these are bona
fide closures on different topologies$.$ For ${k=1}$ these identities survive
as inclusions only$.$ Indeed,
$({\X\\A})^{-1}\sse({\X\\A})^{-2}\!={\X\\A}^{\circ 2}\!={\X\\A}^{\circ 1}$ and
$({\X\\A})^{\circ 1}\!=({\X\\A})^{\circ 2}\!={\X\\A}^{-2}\sse{\X\\A}^{-1}\!.$
A set is $k$-{\it nowhere dense}\/ if its $k$-closure has empty $k$-interior
(i.e., $(A^{-k})^{\circ k}=\void).$ The notions of $1$-open and
$2$-open coincide, but since $A^{-1}$ may be properly included in
$A^{-2}\kern-1pt$, it may happen
$\void=(A^{-1})^{\circ 1}\kern-1pt\subset\kern-1pt(A^{-2})^{\circ 2}\kern-1pt.$
The next results will be required later.

%%%%%%%%%%%%%%%%%%%%%%%%%%% PROPOSITION 4.2
\vskip4pt\noi
\begin{proposition}
Suppose\/ ${B\kern-1pt\sse\kern-1ptA}$ are nonempty subsets of a normed
space\/ $\X.$ If the difference\/ ${A\\B}$ lies in a finite union of
one-dimensional subspaces of
$\X$, then
$$
A^{-k}=\X
\;\limply\;
B^{-k}=\X
\;\;\;\hbox{for every}\;\;
k=0,1,2,3.
$$
\end{proposition}

\begin{proof}
If ${B=A}$ the result is tautological$.$ First suppose
${\void\ne B\subset A\sse\X}$ are such that ${A\\B\sse[u]}$, where $[u]$ is an
arbitrary one-dimensional subspace of $\X$ (spanned by a singleton $\{u\}$
at a nonzero vector ${u\in\X}).$ We split the proof into two parts.

\vskip6pt\noi
(a) Consider the $k$-topologies for $k={0,2,3}.$ The identity
${B^{-k}\!=A^{-k}\!}$ holds if and only if the difference ${A\\B}$ is
$k$-nowhere dense, that is, if and only if $((A\\B)^{-k})^{\circ k}=\void.$
Since $((A\\B)^{-k})^{\circ k}\kern-1pt\sse([u]^{-k})^{\circ k}\kern-1pt$,
then $([u]^{-k})^{\circ k}\kern-1pt=\void$ implies
$B^{-k}\kern-1pt=A^{-k}\kern-1pt.$ If $k=0$ (norm topology), then
$([u]^{-0})^{\circ0}\kern-1pt=\void$ because $[u]$ is $0$-closed
and $[u]^{\circ 0}\kern-1pt=\void.$ If $k=$ ${2,3}$, then
$([u]^{-k})^{\circ k}\kern-1pt=\void$ as well, since in a finite-dimensional
subspace weak and strong (and the intermediate weak sequential) topologies
coincide (see, e.g., \cite[Proposition 2.5.13, 2.5.22]{Meg}), and so $[u]$
$0$-closed means $[u]$ $k$-closed and
$[u]^{\circ k}\kern-1pt=[u]^{\circ 0}\kern-1pt$ for $k={2,3}$ on the
one-dimensional subspace $[u].$ Then ${B^{-k}\kern-1pt=A^{-k}\kern-1pt}.$
Hence $A^{-k}\kern-1pt=\X$ implies $B^{-k}\kern-1pt=\X$ for $k={0,2,3}$
whenever ${A\\B}$ lies in a one-dimensional space.

\vskip6pt\noi
(b) For ${k=1}$ proceed as follows$.$ Suppose $A^{-1}\kern-1pt=\X.$ Take an
arbitrary ${x\in\X}.$ Thus there is an $A$-valued sequence $\{a_n\}$ such
that ${a_n\kern-1pt\wconv x}.$ If ${a_n\kern-1pt\sconv x}$, then
${x\in A^{-0}\kern-1pt}.$ But ${B^{-0}\!=A^{-0}\kern-1pt}$ by item (a)$.$
Hence ${x\in B^{-0}\kern-1pt}$ and so ${x\in B^{-1}\kern-1pt}$ (since
${B^{-0}\sse B^{-1}}).$ Therefore ${B^{-1}\kern-1pt=\X}$ (i.e.,
${\X\sse B^{-1}}).$ On the other hand, if ${a_n\kern-1pt\notsconv x}$, then
$\{a_n\}$ is not eventually in ${A\\B\sse[u]}$ (where weak and strong
convergence coincide as we saw above)$.$ Then there is a subsequence $\{b_n\}$
of $\{a_n\}$ for which ${b_n\kern-1pt\not\in A\\B\sse[u]}$ for every $n$, and
so $\{b_n\}$ is a $B$-valued sequence such that ${b_n\wconv x.}$ Therefore
${B^{-1}\kern-1pt=\X}$.

\vskip6pt\noi
Thus ${\void\ne B\subset\kern-1ptA\sse\X}$ and ${A^{-k}\!=\X}$ imply
${B^{-k}\!=\X}$ for every $k={0,1,2,3}$ if ${A\\B}$ lies in a
one-dimensional space, and the same line of reasoning holds if $A\\B$ lies
in a finite union of one-dimensional spaces (properly included in $\X$).
\end{proof}

\vskip4pt\noi
{\it Remark}\/.
Proposition 4.2 still holds if\/ the difference ${A\\B}$ lies in a
finite union of proper finite-dimensional subspaces of\/ $\X$.

%%%%%%%%%%%%%%%%%%%%%%%%%%% PROPOSITION 4.3
\vskip4pt\noi
\begin{proposition}
If $A$ is a set in a normed space $\X$ and ${L\in\BX}$, then
$$
L(A^{-k})\sse L(A)^{-k}
\;\;\hbox{for every}\;\;
k=0,1,2,3.
$$
\end{proposition}

\begin{proof}
$\!$The inclusion holds for a continuous map $\kern-1ptL\kern-1pt$ between
topological spaces (see, e.g., \cite[Problem 3.46]{EOT})$.$ If $L$ is a
linear continuous map between normed spaces, then weak and strong (and the
intermediate weak sequential) continuities coincide (see, e.g.,
\cite[Theorem 2.5.11]{Meg})$.$ Thus the inclusion holds for $k={0,2,3}$
whenever ${L\in\BX}$ (i.e., $L$ is continuous in the (norm) $0$-topology)$.$
The case of ${k=1}$ requires a separate proof$.$ If ${x\in L(A^{-1}})$, then
${x=La}$ for some ${a\in A^{-1}}.$ But ${a\in A^{-1}}$ if and only if
$f(a)=\lim_jf(a_j)$ with ${a_j\in A}$ for every ${f\in\X^*}\!.$ Since
${f\circ L=L^*\!f\in\X^*}\!$ for the normed-space adjoint $L^*$ (see, e.g.,
\cite[Section 3.2]{Sch} ) we get with $g={L^*\!f\in\X^*\!}$
$$
f(x)=f(La)=(L^*\!f)(a)=g(a)={\lim}_jg(a_j)={\lim}_j(L^*\!f)(a_j)
={\lim}_jf(La_j)
$$
for every ${f\kern-1pt\in\kern-1pt\X^*}\!.$ Since ${L(a_j)\in L(A)}$, then
${x\kern-1pt\in\kern-1ptL(A)^{-1}}\!.$ Thus ${L(A^{-1})\sse L(A)^{-1}}\!.$
\end{proof}

\vskip4pt
For each $k$ a vector ${y\in\X}$ is $k$-{\it supercyclic}\/ for an operator
${T\in\BX}$ (and $T$ is a $k$-{\it supercyclic}\/ operator) if the projective
orbit $\Oe_T([y])$ is $k$-dense in $\X.$ Thus let
\vskip4pt\noi
\begin{description}
\item{}
$Y_0$ be the collection of all strongly supercyclic vectors for $T$,
\vskip2pt
\item{}
$Y_1$ be the collection of all weakly l-sequentially supercyclic
vectors for $T$,
\vskip2pt
\item{}
$Y_2$ be the collection of all weakly sequentially supercyclic
vectors for $T$,
\vskip2pt
\item{}
$Y_3$ be the collection of all weakly supercyclic vectors of $T$.
\end{description}
\vskip2pt\noi
So $T$ {\it is\/ $k$-supercyclic if and only if}\/ ${Y_k\ne\void}.$ According
to Proposition 2.1(b,c),
$$
Y_0\sse Y_1\sse Y_2\sse Y_3,                                     \eqno{(1)}
$$
\vskip2pt\noi
$$
Y_k^{-0}\sse Y_k^{-1}\sse Y_k^{-2}\sse Y_k^{-3}
\quad
\hbox{for every}\;\;k=0,1,2,3.                                   \eqno{(2)}
$$
\vskip4pt

%%%%%%%%%%%%%%%%%%%%%%%%%%%%%%%%%%%%%%%%%%%%%%%%%%%%%%%%%% SECTION 5
\section{Denseness of Supercyclic Vectors}

The punctured projective orbit of a vector $y$ in a normed space $\X$ under
an operator ${T\in\BX}$ is the projective orbit of $y$ excluding the origin,
$$
\Oe_T([y])\\\0
=\big\{\alpha T^ny\in\X\!:\;\alpha\in\FF\\\0,\;n\in\NN_0\big\}\\\0.
$$
By definition of $k$-supercyclicity, for each $k={0,1,2,3}$
$$
y\in Y_k\;\;\iff\;\;(\Oe_T([y])\\\0)^{-k}=\X.                    \eqno{(3)}
$$

%%%%%%%%%%%%%%%%%%%%%%%%%%% LEMMA 5.1
\vskip2pt\noi
\begin{lemma}
For every $k={0,1,2,3}$
$$
y\in Y_k\;\;\limply\;\;(\Oe_T([y])\\\0)\sse Y_k.
$$
\end{lemma}

\vskip0pt
\begin{proof}
Take an operator ${T\in\BX}$ and an arbitrary $k={0,1,2,3}.$ Suppose
${y\in Y_k}$ (i.e., $\Oe_T([y])^{-k}\!=\X$) and take an arbitrary
${z\in\Oe_T([y])\\\0}.$ Thus $z={\gamma\,T^my}$ for some nonzero scalar
${\gamma\in\FF}$ and some nonnegative integer ${m\in\NN_0}$, and hence
$$
\Oe_T([z])=\Oe_T([T^my])
={\bigcup}_{n\ge0}[T^nT^my]
={\bigcup}_{n\ge0}[T^{n+m}y]
={\bigcup}_{n\ge m}[T^ny].
$$
So $\void\ne\Oe_T([z])\sse\Oe_T([y])$ and the difference
$\Oe_T([y])\\\Oe_T([z])$ lies in a finite union
${\bigcup}_{n\in[0,m-1]}\,[T^ny]$ of one-dimensional subspaces of $\X.$ Since
${\Oe_T([y])^{-k}\!=\X}$, then $\Oe_T([z])^{-k}\!=\X$ by Proposition 4.2$.$
Thus ${z\in Y_k}.$ Therefore ${(\Oe_T([y])\\\0)\sse Y_k}.$
\end{proof}

%%%%%%%%%%%%%%%%%%%%%%%%%%% REMARK 5.1
\vskip4pt\noi
\begin{remark}
Take an arbitrary index $k={0,1,2,3}.$ By Lemma 5.1 and (3),
$$
\hbox{$Y_k\ne\void\;\;\limply\;Y_k^{-k}\!=\X$}.
$$
Moreover, using (1) and (2) this can be readily extended to
$$
Y_k\ne\void
\;\;\limply\;
Y_i^{-j}\!=\X
\quad\hbox{for every}\;\,i,j\in[k,3].
$$
In particular (for ${k=0}$), ${Y_0\ne\void}$ implies $Y_1^{-0}=\X.$ This,
however, is not enough to answer the question whether, for instance,
$$
\void=Y_0\subset Y_1\ne\void
\;\;\;\query\;\;
Y_1^{-0}=\X.
$$
Along this line it was proved in \cite[Proposition 2.1]{San1} that
${Y_3\ne\void}\!\limply\!{Y_3^{-0}\!=\X}.$ So
$$ 
Y_3\ne\void
\;\;\limply\;
Y_3^{-j}\!=\X
\quad\hbox{for all}\;\,j
$$
by (2), which still does not answer the above question$.$ This is
extended in the next theorem (using an argument similar to the one
in $\kern-1pt$\cite[$\kern-1pt$Proposition 2.1]{San1}) to show that
$$
Y_k\ne\void
\;\;\limply\;
Y_k^{-0}\!=\X
\quad
\hbox{for every}\;\;k=0,1,2,3.
$$
In particular, for ${k=1}$ this represents a real and useful gain over the
previously known results along this line, answering the above question, and
leading to a general case for any nonempty set of supercyclic vectors with
respect to any notion of denseness (including the nontopological
$1$-denseness)$.$
\end{remark}

\vskip4pt
Regarding the above remark and the next theorem, the condition $Y_k\ne\void$
is fulfilled whenever ${Y_\ell\ne\void}$ for some ${\ell\in[0,k]}$ by (1).

%%%%%%%%%%%%%%%%%%%%%%%%%%% THEOREM 5.1
\vskip4pt\noi
\begin{theorem}
Take an operator\/ $T\kern-1pt$ on a complex normed space\/ $\X.$ For\/
$k=0,1,2,3$
$$
Y_k\ne\void
\;\;\limply\;
Y_i^{-j}\!=\X
\;\;\hbox{for every\/ $i\in[k,3]$ and all\/ $j=0,1,2,3$}.
$$
\end{theorem}

\vskip0pt\noi
\begin{proof}
Take ${T\in\BX}$.

\vskip4pt\noi
{\it Claim}\/$.$
$\;Y_k\ne\void\;\;\limply\;\;Y_k^{-0}\!=\X$
\;for every\; $k=0,1,2,3$.

\vskip4pt\noi
{\it Proof}\/$.$
Let $k$ be an arbitrary index in ${[0,3]}.$ Suppose ${Y_k\ne\void}$ and take
any ${y\in Y_k}.$ Let $p$ be an arbitrary nonzero polynomial$.$ Then
$$
p(T)(\X)=p(T)(\Oe_T([y])^{-k})
\,\sse\,\big(p(T)(\Oe_T([y]))\big)^{-k}\!=\Oe_T([p(T)y])^{-k}\!,
$$
where the above inclusion holds for each $k$ by Proposition 4.3$.$
Thus if $p(T)(\X)^{-k}\!=\X$, then $\Oe_T([p(T)y])^{-k}\!=\X$ and so
${p(T)y\in Y_k}$ by (3)$.$ In other words, if the range of $p(T)$ is $k$-dense
in $\X$, then the vector $p(T)y$ is $k$-supercyclic whenever $y$ is:
$$
y\in Y_k
\quad\hbox{and}\quad
p(T)(\X)^{-k}=\X\quad\limply\quad p(T)y\in Y_k.                    \eqno{(*)}
$$
Now take the dual $\X^*\!$ of the complex normed space $\X$, let
${T^*\!\in\B[\X^*]}$ stand for the normed-space adjoint of ${T\in\BX}$, and
let $\sigma_{\kern-1ptP}(T^*)$ be the point spectrum (i.e., the set of all
eigenvalues) of $T^*\!.$ According to \cite[Lemma 2]{Per} the range of $p(T)$
is dense in a complex locally convex space if and only if all eigenvalues of
$T^*$ are not zeros of $p.$ The strong (norm) topology of a normed space
yields a locally convex space$.$ Thus the range of $p(T)$ is strongly
dense if and only if all eigenvalues of $T^*$ are not zeros of $p.$ But range
is a linear manifold, thus a convex set, and hence all $k$-closures coincide
(cf$.$ Proposition 2.1(d)) so that
$$
p(T)(\X)^{-k}\!=\X
\;\;\iff\;\;
p(\lambda)\ne\0\;
\hbox{for all}\;\lambda\in\sigma_{\kern-1ptP}(T^*).               \eqno{(**)}
$$
Next for any ${y\in\X}$ consider the sets
$$
P_T(y)=\big\{p(T)y\in\X\!:\,\hbox{$p$ is a polynomial}\big\}=\span\Oe_T(y),
$$
\vskip-2pt\noi
$$
P'_T(y)=\big\{p(T)y\in P_T(y)\!:\,p(\lambda)\ne0
\;\;\hbox{for all}\;\;\lambda\in\sigma_{\kern-1ptP}(T^*)\big\}.
$$
\vskip2pt\noi
According to ($**$) and ($*$),
$$
y\in Y_k
\;\;\limply\;\;
P'_T(y)=\big\{p(T)y\in P_T(y)\!:\;p(T)(\X)^{-k}\!=\X\big\}\sse Y_k.
                                                                \eqno{\rm(i)}
$$
We will show that $P'_T(y)$ is strongly dense in $\X$, and consequently
$Y_k$ is strongly dense in $\X.$ First consider the following auxiliary
result.
$$
\hbox{If\/ $T$ is $3$-supercyclic, then ${\#\sigma_{\kern-1ptP}(T^*)\le1}$}
$$
where $\#$ stands for cardinality$.$ This was verified for supercyclic
operators on a Hilbert space in \cite[Proposition 3.1]{Her}, extended to
supercyclic operators on a normed space in \cite[Theorem 3.2]{AB}, and
further extended to supercyclic operators on a locally convex space in
\cite[Lemma 1, Theorem 4]{Per}$.$ But the weak topology of a normed space is
a locally convex subtopology of the locally convex norm topology --- see e.g.,
\cite[Theorems 2.5.2 and 2.2.3]{Meg}$.$ Then the latter extension holds in
particular on a normed space under the weak topology, thus including weakly
supercyclic operators on a normed space$.$ Since
${Y_k\;\ne\void\limply Y_3\ne\void}$ for any $k$ by (1) we get
$$
Y_k\ne\void
\;\;\limply\;\,
\#\sigma_{\kern-1ptP}(T^*)\le1.                                \eqno{\rm(ii)}
$$
Moreover,
$\Oe_T([y])\kern-1pt=\kern-1pt\Oe_T(\span\{y\})\kern-1pt
\sse\kern-1pt\span\Oe_T(y)\kern-1pt=\kern-1ptP_T(y).$
So
$\big(\span\Oe_T(y)\big)^{_{\scriptstyle{-0}}}\!=\kern-1pt\X$ whenever
$\Oe_T([y])^{-k}\!=\X$
(i.e., whenever ${y\in Y_k}$) for an arbitrary $k$ since $\span\Oe_T(y)$ is
convex (cf$.$ Proposition 2.1(d)\/)$.$ Hence
$$
y\in Y_k
\;\;\limply\;\,
P_T(y)^{-0}=\X.                                               \eqno{\rm(iii)}
$$
If ${\#\sigma_{\kern-1ptP}(T^*)=0}$ and ${y\in Y_k}$, then $P'_T(y)=P_T(y)$
and so by (iii)
$$
y\in Y_k
\;\;\hbox{and}\;\;
\#\sigma_{\kern-1ptP}(T^*)=0
\;\;\limply\;\,
P'_T(y)^{-0}=\X.
$$
On the other hand, if ${\#\sigma_{\kern-1ptP}(T^*)=1}$ (i.e., if
$\sigma_{\kern-1ptP}(T^*)=\{\lambda_0\}$ for some ${\lambda_0\in\CC}$), then
$P'_T(y)=\{{p(T)y\in P_T(y)\!}:{p(\lambda_0)\ne0}\}$ is dense in $P_T(y)$ in
the norm topology, and $P_T(y)$ is dense in $\X$ in the norm topology by
(iii) whenever ${y\in Y_k}.$ Thus
$$
y\in Y_k
\;\;\hbox{and}\;\;
\#\sigma_{\kern-1ptP}(T^*)=1
\;\;\limply\;\,
P'_T(y)^{-0}=\X.
$$
Hence by (i), (ii), and the preceding two implications we get the claimed
result:
$$
y\in Y_k
\;\;\limply\;\,
P'_T(y)^{-0}=\X
\;\;\limply\;\,
Y_k^{-0}=\X.\!\!\!\qed
$$

\vskip4pt\noi
If\/ ${Y_k^{-0}\kern-2pt=\kern-1pt\X}$, then
$\kern-1pt{Y_k^{-j}\kern-2.5pt=\kern-1pt\X}$ for
${j\kern-1pt=\kern-1pt0,1,2,3}$ by (2)$.$ So
$\kern-1pt{Y_i^{-j}\kern-2.5pt=\kern-1pt\X}$ for ${i\kern-1pt\ge\kern-1ptk}$
by (1)$.\!$
\end{proof}

%%%%%%%%%%%%%%%%%%%%%%%%%%% REMARK 5.2
\vskip4pt\noi
\begin{remark}
Propositions 4.1 to 4.3, besides being required for proving Theorem 5.1, may
be seen as relevant by themselves in the fol\-lowing sense$.$ Proposition 4.1
applies standard convergence techniques for building sequential topological
spaces from topological spaces, which is needed to support the topological
case of ${k\kern-1pt=\kern-1pt2}.$ Propositions 4.2 and 4.3 are also standard
for the topological cases of ${k\kern-1pt=\kern-1pt0,2,3}$, but they seem new,
nontrivial, and relevant for the nontopological case of
${k\kern-1pt=\kern-1pt1}.$ In fact, the emphasis of the whole paper,
especially the emphasis of the main result in Theorem 5.1, is on the
nontopological case of ${k\kern-1pt=\kern-1pt1}$, which is here brought
together with the other topological cases$.$ In particular, Theorem 5.1
answers the question posed in Remark 5.1: {\it the set of weakly l-essentially
supercyclic vectors is dense in the norm topology if it is not empty, even if
there is no supercyclic vector}\/,
$$
\void=Y_0\subset Y_1\ne\void
\;\;\;\limply\;\;
Y_1^{-0}=\X.
$$
\end{remark}
\vskip4pt

%%%%%%%%%%%%%%%%%%%%%%%%%%%%%%%%%%%%%%%%%%%%%%%%%%%%%%%%%% SECTION 6
\section{Weak Supercyclicity and Stability}

An immediate consequence of Theorem 5.1 says, roughly speaking, that {\it if
an arbitrary set of supercyclic vectors is not empty, then it is dense with
respect to any notion of denseness}\/. This is properly stated as follows.

%%%%%%%%%%%%%%%%%%%%%%%%%%% COROLLARY 6.1
\vskip4pt\noi
\begin{corollary}
Consider the sets of strongly supercyclic, weakly l-sequentially supercyclic,
weakly sequentially supercyclic, and weakly supercyclic vectors for an
arbitrary normed-space operator$.$ All these sets are dense in the normed
space, regardless the notion of denseness one is considering, provided they
are nonempty.
\end{corollary}

\vskip2pt
An operator ${T\kern-1pt\in\kern-1pt\BX}$ is {\it power bounded}\/ if
$\sup_{n\ge0}\|T^n\|\kern-1pt<\kern-1pt\infty.$ It is {\it strongly stable}\/
if ${T^nx\sconv0}$ for every ${x\in\X}$, and {\it weakly stable}\/ if
${T^nx\wconv0}$ for every ${x\in\X}$.

%%%%%%%%%%%%%%%%%%%%%%%%%%% REMARK 6.1
\vskip4pt\noi
\begin{remark}
If $\{T_n\}$ is a bounded sequence of operators on a normed space $\X$ and
$\{T_ny\}$ converges strongly to $Ty$ for some ${T\!\in\kern-1pt\BX}$ for
every $y$ in a strongly dense subset $Y$ of $\X$, then $\{T_nx\}$ converges
strongly to $Tx$ for every ${x\in\X}.$ In particular, since ${Y_k^{-0}=\X}$
for each $\,k={0,1,2,3}\,$ whenever ${Y_k\ne\void}\,$ by Theorem 5.1, we get:
\vskip6pt\noi
{\narrower\narrower
{\it If\/ ${T\kern-1pt\in\kern-1pt\BX}$ is a power bounded\/ $k$-supercyclic
operator for an arbitrary\/ $k={0,\kern-.5pt1,\kern-.5pt2,\kern-.5pt3}$, and
if\/ ${T^ny\kern-1pt\sconv0}$ for every\/ ${y\kern-1pt\in\kern-1ptY_k}$,
then\/ $T$ is strongly stable}\/.
\vskip6pt}\noi
It was proved in \cite[Theorem 2.2]{AB} that {\it if a power bounded operator
is strongly supercyclic, then it is strongly stable}\/$.$ Thus the above
strong stability result holds for ${k\kern-1pt=\kern-1pt0}$ without any
additional assumption$.$ {\it Does any form of weak supercyclicity {\rm (i.e.,
$k$-supercyclicity for $k={1,2,3}$)} imply weak stability for power bounded
operators}\/$?$ The question was posed and investigated in \cite{KD1} and
remains unanswered even if Banach-space power bounded operators are
restricted to Hilbert-space contractions.
\end{remark}

\vskip2pt
Here is the weak version of the above italicized displayed statement.

%%%%%%%%%%%%%%%%%%%%%%%%%%% COROLLARY 6.2
\vskip4pt\noi
\begin{corollary}
If\/ ${T\in\BX}$ is a power bounded\/ $k$-supercyclic operator for an
arbitrary\/ ${k=0,1,2,3}$, and if\/ ${T^ny\wconv0}$ for every\/ ${y\in Y_k}$,
then\/ $T$ is weakly stable.
\end{corollary}

\begin{proof}
This is a consequence of Theorem 5.1 and the following result.

\vskip4pt\noi
{\it Claim}\/.
If $\{T_n\}$ is a bounded sequence of operators on a normed space $\X$ and
$\{T_ny\}$ converges weakly to $Ty$ for some ${T\!\in\kern-1pt\BX}$ for
every $y$ in a strongly dense subset $Y$ of $\X$, then $\{T_nx\}$ converges
weakly to $Tx$ for every ${x\in\X}.$

\vskip4pt\noi
{\it Proof}\/$.$
Take any ${x\kern-1pt\in\kern-1pt\X}.$ If ${Y^{-0}\!=\X}$, then there exists
a $Y\!$-valued sequence $\{y_m\}$ converging strongly to $x$, which means
$\|y_m-x\|\to0.$ Suppose ${T_ny\wconv Ty}$ for every ${y\in Y}\!$, which means
${f(T_ny)\to f(Ty)}$ for every ${f\in\X^*\!}$ and every ${y\in Y}\!$, and so
${|f(T_ny_m-Ty_m)|\to0}$ for every $m.$ Thus since for every
${f\kern-1pt\in\kern-1pt\X^*\!}$ and every ${x\kern-1pt\in\kern-1pt\X}$
\begin{eqnarray*}
|f\big((T_n-T)\,x)|
&\kern-6pt\le\kern-6pt&
|f((T_n-T)\,(y_m-x))|+|f((T_n-T)\,y_m)|                                 \\
&\kern-6pt\le\kern-6pt&
\|f\|\,({\sup}_n\|T_n\|+\|T\|)\kern1pt\|y_m-x\|+|f(T_ny_m-Ty_m)|,
\end{eqnarray*}
then we get the claimed assertion: ${T_nx\wconv Tx}$
for every ${x\in\X}.\!\!\!\qed$

\vskip4pt\noi
Suppose the power sequence of a power bounded operator on a complex normed
space is weakly stable over a set of $k$-supercyclic vectors $Y_k.$
$\kern-.5pt$Since $\kern-1pt{Y_k^{-0}\kern-2pt=\kern-1pt\X}$ for every
$k\kern-1pt=\kern-1pt{0,1,2,3}$ if ${Y_k\kern-1pt\ne\kern-1pt\void}$ by
$\kern-1pt$Theorem 5.1, the above claim ensures the stated result$.\kern-4pt$
\end{proof}

\vskip2pt
In particular, if $T$ is power bounded, ${Y_1\ne\void}$, and ${T^ny\wconv0}$
for every $y$ in the set $Y_1$ of all weakly l-sequentially supercyclic
vectors, then $T$ is weakly stable.

%%%%%%%%%%%%%%%%%%%%%%%%%%%%%%%%%%%%%%%%%%%%%%%%%%%%%%%%%% REFERENCES
\vskip-10pt\noi
\bibliographystyle{amsplain}

\end{document}